\documentclass[reqno,12pt]{amsart}

\usepackage{graphicx}
\usepackage{amssymb}
\usepackage{amsmath}
\usepackage{enumitem}

\date{}

\theoremstyle{plain}
\newtheorem{theorem}{Theorem}
\newtheorem{corollary}{Corollary}

\theoremstyle{remark}

\newtheorem*{remark}{Remark}

\def\N{\mathbb N}
\def\Z{\mathbb Z}
\def\R{\mathbb R}

\def\phi{\varphi}
\def\leq{\leqslant} 
\def\geq{\geqslant} 
\def\x{\times}

\title[On families of fibred knots with equal Seifert forms]{On families of fibred knots with\\ equal Seifert forms}

\author{Filip Misev}
\address{Institut de Math\'ematiques I2M, Universit\'e Aix-Marseille, 39 rue F.~Joliot Curie, 13453 Marseille, France}
\email{filip.misev@univ-amu.fr}
\thanks{The author is supported by the Swiss National Science Foundation (\#168676).} 

\begin{document}
\begin{abstract}
For every genus $g\geq 2$, we construct an infinite family of strongly quasipositive fibred knots having the same Seifert form as the torus knot $T(2,2g+1)$. In particular, their signatures and four-genera are maximal and their homological monodromies (hence their Alexander module structures) agree. On the other hand, the geometric stretching factors are pairwise distinct and the knots are pairwise not ribbon concordant.
\end{abstract}

\maketitle
\thispagestyle{empty}

\section{Introduction and summary of results}
Baker proved in a recent note that two strongly quasipositive fibred knots $K_0$, $K_1$ must be equal if $K_0\#(-K_1)$ is a ribbon knot \cite{Bak}. In~view of Fox's famous slice-ribbon question, this leads to two equally intriguing alternative statements: either there exists a fibred knot of the form $K_0\#(-K_1)$ which is slice but not ribbon, or every knot concordance class contains at most one strongly quasipositive fibred knot.

We describe a simple and general construction proposed by Baader (compare \cite{Ba1}), which gives rise to infinite families of strongly quasipositive fibred knots having the same Seifert form. As~a special case in which it is relatively easy to study the constructed knots in some detail, we obtain the following theorem.

\begin{theorem} \label{thm1}
Let $g\geq 2$. There exists an infinite family of (pairwise nonisotopic) genus $g$ knots $K_n\subset S^3$, $n\in\N$, such that for all $n\in\N$
\begin{enumerate}
\item $K_n$ is strongly quasipositive and fibred,
\item $K_n$ has the same Seifert form as the torus knot $T(2,2g+1)$.
\end{enumerate}
\end{theorem}

To distinguish the constructed knots, we compare the geometric stretching factors of their monodromies, which can be calculated explicitly for $g=2$. The following corollary is an immediate consequence.

\newpage

\begin{corollary} \label{cor1}
Let $g\geq 2$. There exist infinitely many hyperbolic strongly quasipositive fibred knots of genus $g$, having
\begin{enumerate}
\item maximal signature,
\item periodic homological monodromy,
\item isomorphic Alexander module structures.
\end{enumerate}
\end{corollary}

Strongly quasipositive links were introduced and studied by Rudolph \cite{Ru1} and generalise links of singularities, positive braid closures, and positive links \cite{Ru2}. They are characterised as the boundaries of incompressible subsurfaces of positive torus link fibres \cite{Ru0}, form a rather large class of knots representing every possible Seifert form \cite[\S 3]{Ru} and behave naturally with respect to fibredness: connected sum, plumbing and positive cabling preserve strong quasipositivity.

Families of fibred knots with equal Seifert forms have been studied before in connection with (ribbon) concordance. For example, Bonahon used Stallings twists to construct a family of fibred genus two knots $K_n$, $n\in\N$, whose homological monodromies are pairwise conjugate and such that $K_n\#(-K_m)$ is not a ribbon knot for $n\neq m$ (compare~\cite{Bo1}, see also \cite{Bo2}). However, these knots are not strongly quasipositive. In~fact, a quasipositive surface cannot contain an essential zero-framed annulus \cite{Ru3}, which is required for a Stallings twist. To~prove that the knots $K_n\#(-K_m)$ are not ribbon, Bonahon applied the following criterion of Casson and Gordon \cite[Theorem~5.1]{CaGo}.
\begin{center}
{\em A fibred knot in a homology 3-sphere is homotopically ribbon if and only if its closed monodromy extends over a handlebody.}
\end{center}

The notion of a knot being homotopically ribbon is a weakened version of ribbonness introduced by Casson and Gordon. If $S$ is a compact orientable surface with boundary, a diffeomorphism $\phi:S\to S$ (fixing the boundary of $S$ pointwise) is said to extend over a handlebody if and only if there exists a three-dimensional handlebody $W$ such that $S\subset\partial W$, $\partial W\setminus S$ is a union of discs, and $\phi$ is the restriction of a diffeomorphism of $W$.

\begin{corollary}
Let $K_n$, $n\in\N$, be the family of knots constructed in the proof of Theorem~\ref{thm1}. The following holds for all $i,j\in\N$, $i\neq j$.
\begin{enumerate}
\item $K_i$ and $K_j$ are not ribbon concordant,
\item $K_i$ and $K_j$ are not homotopy ribbon concordant,
\item the closed monodromy of $K_i\#(-K_j)$ does not extend over a handlebody,
\item the smooth four-ball genus of $K_i\#(-K_j)$ is at most one.
\end{enumerate}
\end{corollary}

\section{Families of fibred knots}

We first describe an explicit construction of a family of knots $K_n$, $n\in\N$, having the properties stated in Theorem~\ref{thm1}. To prove that the knots are indeed pairwise nonisotopic, we compute the geometric dilatation $\lambda_n\in\R$ of the monodromy of $K_n$. For this purpose, we choose the construction in such a way that the monodromy is represented by a composition of two multi-twists, that is, products of Dehn twists on sets of disjoint curves.

\begin{proof}[Proof of Theorem~\ref{thm1} -- Construction of the knots $K_n$]
Let $S\subset S^3$ be the fibre surface of the $T(2,2g)$ torus link, viewed as a plumbing of $2g-1$ positive Hopf bands $H_1,\ldots,H_{2g-1}$ according to the tree \includegraphics{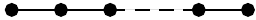}. We choose the plumbing order as follows. First let $S'=H_1\#\ldots\#H_g$ be the connected sum of the first $g$ bands. The monodromy of $S'$ is a composition of positive Dehn twists along $g$ pairwise disjoint simple closed curves, namely the core curves $\alpha_1,\ldots,\alpha_g$ of $H_1,\ldots,H_g$. Denote $\beta_1,\ldots,\beta_{g-1}$ the cores of the remaining Hopf bands $H_{g+1},\ldots,H_{2g-1}$, respectively. The latter may now be plumbed to $S'$ from below such that $\alpha_1,\beta_1,\alpha_2,\beta_2,\ldots,\alpha_{g-1},\beta_{g-1},\alpha_g$ form a chain, that is, $\beta_i$ intersects each of its neighbours $\alpha_i$ and $\alpha_{i+1}$ transversely in one point (for all $i=1,\ldots,g-1$), and there are no other intersections (see Figure~\ref{fig:Sg}).
\begin{figure}[ht]
\includegraphics[scale=0.5]{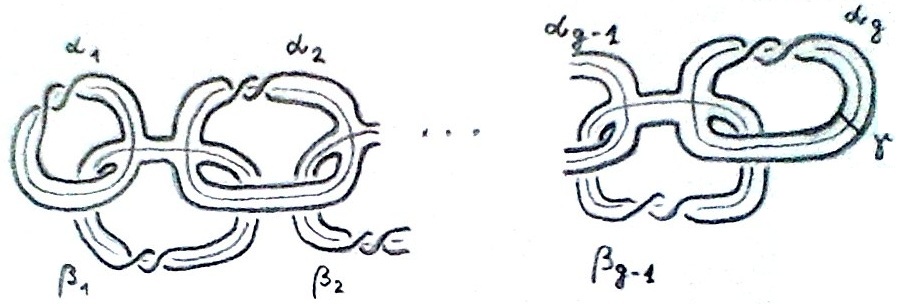}
\caption{The fibre surface $S$ of the torus link $T(2,2g)$ and core curves of the plumbed Hopf bands.}
\label{fig:Sg}
\end{figure}
Denote $S$ the resulting surface, which is indeed the fibre surface of $T(2,2g)$, since the plumbing order is irrelevant for arborescent Hopf plumbings. Now let $\gamma\subset H_{2g-1}\subset S$ be a proper arc which intersects $\alpha_g$ transversely in one point and does not intersect any of the other curves. Further let $c\subset S$ be the boundary of a (small) regular neighbourhood of $\alpha_{g}\cup\beta_{g-1}$ in $S$ and denote $t_c$ the positive Dehn twist on $c$. Note that $t_c$ acts as the identity on homology since $c$ is nullhomologous in $S$, by construction (see Figure~\ref{fig:S2}).
\begin{figure}[ht]
\includegraphics[scale=0.5]{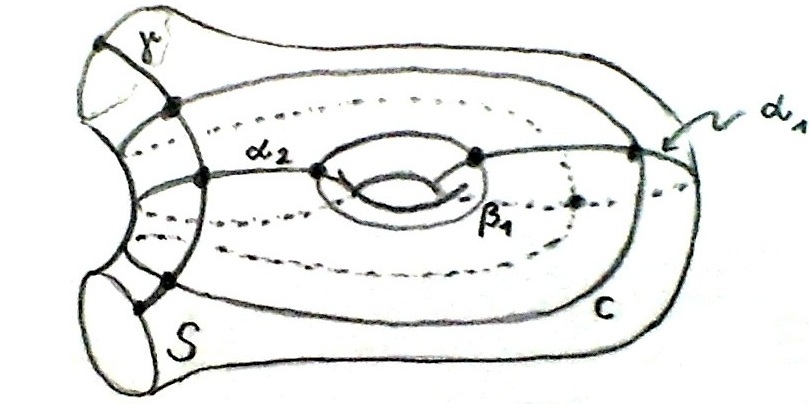}
\caption{Position of the curves in the abstract surface $S$ for $g=2$.}
\label{fig:S2}
\end{figure}
For $n\in\N$, define the proper arc $\gamma_n=t_c^n(\gamma)$. Finally, let $S_n$ be the surface obtained by plumbing a positive Hopf band $H_{2g}$ along $\gamma_n$ to $S$ from below and denote $\beta_{g,n}$ the core curve of $H_{2g}$ in $S_n$. The monodromy $\phi_n:S_n\to S_n$ is given by
\[ \phi_n=(t_{\beta_{g,n}}\circ t_{\beta_{g-1}}\circ\ldots\circ t_{\beta_1})\circ (t_{\alpha_g}\circ\ldots\circ t_{\alpha_1}). \]
In other words, $\phi_n$ is a composition of two positive multi-twists along systems of $g$ disjoint curves.
Each complementary region of the union $\alpha_1\cup\ldots\cup\alpha_g\cup\beta_1\cup\ldots\cup\beta_{g-1}\cup\beta_{g,n}$ is either a boundary parallel annulus or a polygon which has at least three sides (corresponding to sub-arcs of the curves $\alpha_i$, $\beta_i$). Therefore, the curves fill up the surface $S$ and they realise their geometric intersection number (apply the bigon criterion of \cite{FaMa}). Moreover their union is connected. This fits the setting of Thurston's classical construction of invariant measured foliations for products of multi-twists (compare~\cite{Thu}). To compute the geometric dilatation $\lambda_n$ of $\phi_n$, consider the $g\x g$ geometric intersection matrix
\[ N= \left[ \begin{array}{ccccc} 1&&&& \\ 1&1&&& \\ &\ddots &\ddots && \\ &&1&1&4n \\  &&&1&1 \end{array} \right] \]
whose entry $N_{ij}$ is given by the number of intersection points (counted without sign) between $\alpha_i$ and $\beta_j$. Note that $\gamma\cap c$ and $\alpha_{g-1}\cap c$ consist of two points each, which implies that $\gamma_n$ (and therefore $\beta_{g,n}$) intersects $\alpha_{g-1}$ in $4n$ points. Let $\mu_n$ be the largest eigenvalue of the symmetric matrix $NN^\top$, which is of the following form for $g\geq 3$.
\[ NN^\top = \left[ \begin{array}{ccccll} 1&1&&&& \\ 1&2&&&& \\ &&\ddots &1&& \\ &&1&2&1& \\ &&&1&2+16n^2&1+4n \\ &&&&1+4n&2 \end{array} \right] \]
By a classical theorem of Ger\v{s}gorin \cite{Ger}, the eigenvalues of a matrix $A$ are contained in the union of the discs in the complex plane with centers $A_{ii}$ and radii $\sum_{j\neq i} |A_{ij}|$ and if a disc is disjoint from the others, it must contain precisely one eigenvalue. For $A=NN^\top$, we see that the Ger\v{s}gorin disc of radius $2+4n$ centered at the largest diagonal entry $2+16n^2$ is disjoint from all others if $n\neq 0$. Therefore
\[ 16n^2-4n \ \leq \ \mu_n \ \leq \ 16n^2+4n+4,\quad \forall n\neq 0, \]
hence the $\mu_n$ are pairwise distinct and unbounded. By Thurston's construction, the map $\phi_n$ is pseudo-Anosov if and only if the following $2\x 2$ matrix representing $\phi_n$ is hyperbolic.
\[ T_{\underline{\beta}}\cdot T_{\underline{\alpha}}=\left[ \begin{array}{cc} 1&0 \\ -\mu_n^{1/2}&1 \end{array} \right] \cdot \left[ \begin{array}{cc} 1&\mu_n^{1/2} \\ 0&1 \end{array} \right] = \left[ \begin{array}{cc} 1&\mu_n^{1/2} \\ -\mu_n^{1/2} & 1-\mu_n \end{array} \right] \]
Its eigenvalues are the geometric stretching factors $\lambda_n^{\pm 1}$ of $\phi_n$. A quick computation yields
\[ \lambda_n^{\pm 1} = \frac12 (2-\mu_n \mp\sqrt{\mu_n^2-4\mu_n}) \]
(we chose the signs such that $|\lambda_n|\geq 1$ for $n\neq 0$). In particular, $|\lambda_n|\to\infty$ as $n\to\infty$ and the $\lambda_n$ are pairwise distinct. This implies that the knots $K_n\! :=\partial S_n$ are pairwise distinct and hyperbolic for $n\neq 0$. For $n=0$, we have $K_0=T(2,2g+1)$, $\phi_0$ is a periodic mapping class and $\lambda_0$ is a root of unity.

By construction, $S_n$ is a plumbing of $2g$ positive Hopf bands. Since a plumbing of positive Hopf bands is strongly quasipositive \cite{Ru1} and fibred \cite{Sta}, $K_n$ has the same properties, for all $n\in\N$.

It remains to show that the Seifert forms of the surfaces $S_n$ agree, up to the obvious identification of $H_1(S_n,\Z)=H_1(S,\Z)\oplus\langle\beta_{g,n}\rangle$ with $H_1(S_m,\Z)=H_1(S,\Z)\oplus\langle\beta_{g,m}\rangle$. The Seifert forms clearly agree on $H_1(S,\Z)$, since the $S_n$ are all given by plumbing on $S$. Now let $\delta\subset S$ be a closed oriented curve and denote $\delta^*$ a slight push-off along the (positive or negative) direction normal to $S$. We have to prove that the linking numbers of $\beta_{g,n}$ and $\beta_{g,m}$ with $\delta^*$ agree. This follows from the fact that $\beta_{g,n}$ and $\beta_{g,m}$ bound discs whose interiors are disjoint from $S$, hence the linking numbers with $\delta^*$ are given by the algebraic intersection numbers between $\delta$ and $\beta_{g,n}$, $\beta_{g,m}$, respectively (compare Figure~\ref{fig:delta}). The latter agree since the arcs $\gamma_n$ and $\gamma_m$ are homologous in $S$ by construction.
\end{proof}

\begin{figure}[ht]
\includegraphics[scale=0.4]{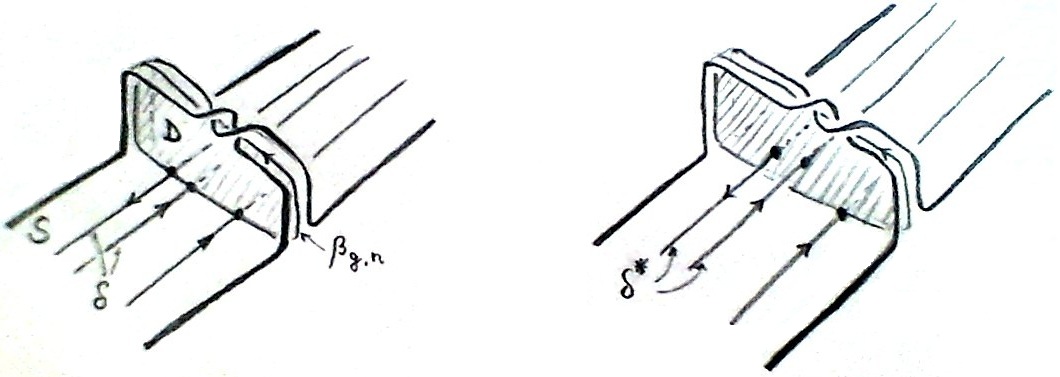}
\caption{The linking number of $\delta^*$ and $\beta_{g,n}$ equals the algebraic intersection number of $\delta$ and $\beta_{g,n}\cap S$.}
\label{fig:delta}
\end{figure}

\setcounter{corollary}{0}
\begin{corollary}
Let $g\geq 2$. There exist infinitely many hyperbolic strongly quasipositive fibred knots of genus $g$, having
\begin{enumerate}
\item maximal signature,
\item periodic homological monodromy,
\item isomorphic Alexander module structures.
\end{enumerate}
\end{corollary}

\begin{proof}
In the proof of Theorem~\ref{thm1} we constructed an infinite family of strongly quasipositive fibred genus $g$ knots $K_n$, $n\in\N$, such that $K_n$ is hyperbolic for $n\neq 0$ and has the same Seifert form as $K_0=T(2,2g+1)$. Since $T(2,2g+1)$ has maximal signature, this immediately implies~$(1)$. The homological monodromy $M$ of a fibre surface is related to its Seifert form $A$ by the formula $M=A^{-\top}A$ (compare \cite[Lemma~8.3]{Sav}). Hence $K_n$ has the same homological monodromy as $K_0$, which is periodic. Since the Alexander module structure of a fibred knot is determined by the homological action of the monodromy, this proves $(2)$ and $(3)$.
\end{proof}

\begin{corollary}
Let $K_n$, $n\in\N$, be the family of knots constructed in the proof of Theorem~\ref{thm1}. The following holds for all $i,j\in\N$, $i\neq j$.
\begin{enumerate}
\item $K_i$ and $K_j$ are not ribbon concordant,
\item $K_i$ and $K_j$ are not homotopy ribbon concordant,
\item the closed monodromy of $K_i\#(-K_j)$ does not extend over a handlebody,
\item the smooth four-ball genus of $K_i\#(-K_j)$ is at most one.
\end{enumerate}
\end{corollary}

\begin{remark}
Note that homotopy ribbon concordance is (a priori) not a symmetric relation. Here, we say that two knots are homotopy ribbon concordant if one of them is homotopy ribbon concordant to the other.
\end{remark}

\begin{proof}
By \cite[Lemma~3.4]{Gor}, two homotopy ribbon concordant strongly quasipositive fibred knots have to be equal, hence (2). By~a theorem of Casson and Gordon, (2) $\Leftrightarrow$ (3) for fibred knots (see \cite[Theorem~5.1]{CaGo}). The implication (2) $\Rightarrow$ (1) holds for all knots. Assertion (4) essentially follows from the fact that the surface $S_i\#(-S_j)$ is given by plumbing two Hopf bands to the surface $S\#(-S)$, which is ribbon. More precisely, we can find a zero-framed unlink with unknotted components $L_1,\ldots,L_{2g}\subset S_i\#(-S_j)$ which are realised as the embedded connected sum of the copies of the curves $\alpha_k$, $\beta_k$ in $S$ and $-S$, respectively. Cut $S_i\#(-S_j)$ along $L_k$ and glue two ribbon discs back in, for each $k$. Push the interior of the resulting surface slightly into the four-ball to obtain a ribbon surface of genus one.
\end{proof}

\section{Further properties of the knots $K_n$}
The subsequent statements refer to the family of knots $K_n$ of a fixed genus $g\geq 2$ constructed in the proof of Theorem~\ref{thm1}.
\begin{enumerate}
\item Since the signature of a knot is a lower bound for the topological four-ball genus, we also have that $g_4^{top}(K_n)=g_4(K_n)=g(K_n)$. In particular, the $K_n$ are not slice (see also \cite{Ru3}).
\item The equality of the Seifert forms also implies that the Levine-Tristram signature functions of the knots $K_n$ agree.
\item The Alexander polynomial of $K_n$ is equal to the Alexander polynomial of $T(2,2g+1)$ and the knots $K_n\#(-K_m)$ satisfy the Fox-Milnor condition.
\item By work of Hedden \cite{Hed}, the concordance invariant $\tau$ is maximal for the knots $K_n$, that is, $\tau(K_n)=g(K_n)=g$. Moreover, the open book associated to $K_n$ supports the tight contact structure of $S^3$, so the geometric monodromy of $K_n$ is right-veering \cite{HKM}.
\item The knots $K_n$ are all prime since their geometric monodromies are irreducible (they are pseudo-Anosov for $n\neq 0$).
\item For $n\neq 0$, $K_n$ cannot be represented by a positive braid. In fact, positive braids of maximal signature have been classified by Baader \cite{Ba2}. They all have periodic (geometric) monodromy. All but finitely many of the $K_n$ (probably all but $K_0$) cannot be represented by a positive knot diagram. This follows from the fact that the number of positive knots of a fixed genus and fixed signature function is finite; compare \cite[Proof of Theorem~1]{BDL}.
\item Similarly, all but finitely many of the $K_n$ are non-alternating, since the number of alternating links of a given Alexander polynomial is finite by work of Stoimenow \cite[Corollary~3.5]{Sto}.
\item Except for a finite number of indices, the fibre surfaces of the $K_n$ cannot be Hopf-plumbed baskets (given by Hopf plumbing along arcs on a fixed disc \cite{Ru4}). In fact, the number of Hopf-plumbed baskets of a given genus is finite. This suggests that baskets should be thought of as rare exceptions among quasipositive fibre surfaces, at least from the point of view of general Hopf plumbing. See also \cite{Ban}.
\end{enumerate}

\end{document}